\documentclass{amsart}
\usepackage{amssymb, amsmath, amsthm, mathtools, hyperref, titlesec}
\makeatletter
\@namedef{subjclassname@2020}{%
	\textup{2020} Mathematics Subject Classification}
\makeatother

\newtheorem{theorem}{Theorem}[section]
\newtheorem{proposition}[theorem]{Proposition}
\newtheorem{lemma}[theorem]{Lemma}
\newtheorem{corollary}[theorem]{Corollary}
\theoremstyle{definition}

\theoremstyle{remark}
\newtheorem*{remark}{Remark}
\newtheorem*{remarks}{Remarks}
\DeclareMathOperator{\rad}{rad} 
 
\newcommand\norm[1]{\left\lVert#1\right\rVert}
\newcommand{\identity}{\mathbf{1}}

\DeclareMathOperator{\orb}{orb}

\DeclarePairedDelimiter\abs{\lvert}{\rvert}
\newcommand{\field}{\mathbb{C}}
\titleformat{\section}{\normalfont\scshape}{\thesection}{1em}{}
\title{Commutativity and Orthogonality of Similarity Orbits in Banach Algebras}

\date{\today}

\author{Muhammad Hassen}
\address{University of Johannesburg, Johannesburg, South Africa}
\email{mhassen@uj.ac.za}
\author{Rudi Brits}
\address{University of Johannesburg, Johannesburg, South Africa}
\email{rbrits@uj.ac.za}
\author{Francois Schulz}
\address{University of Johannesburg, Johannesburg, South Africa}
\email{francoiss@uj.ac.za}

\begin{document}
	
	\begin{abstract}
		For a semisimple unital Banach algebra $ A $ over $ \field $, and elements $a,b\in A,$ we show that the similarity orbits, $ \orb(a)$ and $ \orb(b), $ over the principal component of the invertible group of $ A $ commute precisely when there is at least one nonzero complex number not belonging to the spectrum of any product $ a^\prime b^\prime $—where $ (a^\prime,b^\prime)\in\orb(a)\times\orb(b) $. In this case, the polynomially convex hull of the spectra of the $ a^\prime b^\prime $ is constant. When $ \orb(a)=\orb(b) ,$ then $ a $ is central under the aforementioned assumption—and the result then generalizes part of an old theorem due to J. Zemánek. We show further that the two classical characterizations of commutative Banach algebras via the spectral radius can be algebraically localized in the sense of `local' implies `global'. Thereafter, in Section~\ref{sec3}, we give a (somewhat weaker) localization of the above situation involving spectral perturbation on small neighborhoods in a similarity orbit. Finally, we apply the above results to algebraic elements and idempotents in particular, so that orthogonality of similarity orbits of two idempotents is equivalent to a pair of spectral radius properties. To conclude with, a couple of localization theorems specific to idempotents and algebraic elements are presented. Similar statements to all of the above hold if $ a^\prime b^\prime $ is replaced by $ a^\prime + b^\prime $, $ a^\prime - b^\prime $, or $  a^\prime + b^\prime-a^\prime b^\prime $. 
	\end{abstract}
\maketitle
\section{Introduction}
We will assume throughout this paper that $(A,\norm{\cdot})$ is a semisimple complex and unital Banach algebra with unit denoted by $ \identity $. Furthermore, the group of invertible elements in $ A $ will be denoted by $ G(A) $, with the principal component containing $ \identity $ denoted by $ G_\identity(A) $ (or simply $ G_\identity $ if there is no confusion). We shall use $ \sigma_A $ to denote the spectrum
$$\sigma_A(x)=\{\lambda\in\field\colon\lambda\identity-x\not\in G(A)\},$$
and $ \rho_A $ to denote the spectral radius
$$\rho_A(x)=\sup\{\abs{\lambda}\colon \lambda\in\sigma_A(x)\},$$
of an element $ x\in A $ (and obviously drop the subscript if the context is clear as to which algebra is concerned). The nonzero spectrum of $ x\in A $ will be the set $ \sigma^\prime(x)=\sigma
(x)\setminus\{0\} $, and the polynomially convex hull of the spectrum of $ x $ will be denoted using $ \hat{\sigma}(x) $. Geometrically, $ \hat{\sigma}(x) $ is the spectrum of $ x $ together with all of its holes filled in—which ensures that $ \field\setminus \hat{\sigma}(x) $ is connected. Throughout this paper, we will use $ a\circ b $ to denote the element $ a+b-ab $, for any given $ a,b\in A $. Notice then that $ a\circ b= \identity-(\identity-a)(\identity-b) $. The key ideas in our study will utilise and be concerned with the similarity orbit of an element $ a\in A $, defined as:
$$\orb(a)\coloneqq\left\{waw^{-1}\colon w\in G_\identity\right\}.$$	
Notice that, since $ G_\identity $ is a group, the collection of all similarity orbits of elements of $ A $ forms a partition of $ A $. Due to the nature of the elements of the similarity orbit, Jacobson's Lemma is favourable in the matter of reducing the work required to prove certain facts in our upcoming results. As a reminder, Jacobson's Lemma says that the following is true for all $ a,b\in A $:
$$\sigma^\prime(ab)=\sigma^\prime(ba).$$
Because $ G(A) $ is a group, one can quickly show (without using the above equality) that $ \sigma(a)=\sigma(waw^{-1}) $ for all $ w\in G_\identity(A) $. 

If $ E $ is the set of idempotents in a Banach algebra $ A $ and $ a\in E $, then J. Zemánek showed, amongst a large number of other results, that the connected component of $ E $ containing $ a $ is precisely the similarity orbit of $ a $ ({\cite[Theorem 3.3]{z79}}). Likewise, in 2016, E. Makai together with J. Zemánek in a collaborative effort had proven that if a polynomial $ p $ whose roots all have multiplicity 1 is fixed, then the connected components of the set $E(p)= \{x\in A\colon p(x)=0\} $ are precisely the similarity orbits of each of its members ({\cite[Theorem 3]{makaiZemanek16}}). The connectedness of the similarity orbits in $ E(p) $ is deeply investigated in this latter paper, particularly connectedness via polynomial and polygonal paths. Another important fact that will be used in a few instances is that $a$  \text{ is central }  if and only if   $\orb(a)=\{a\}$. Hence, if $ a^\prime\in\orb(a) $ is central, then $ a=a^\prime $ since $ \orb(a)=\orb(a^\prime)=\{a^\prime\}$.

The last major result in {\cite{z79}} provides ten statements which are equivalent to an idempotent $ a $ being central. In fact, these conditions are all intricately related to the similarity orbit and the spectrum of $ a $. To illustrate the various ideas that will henceforth be discussed, it is worthwhile to list some of the conditions given in {\cite[Theorem 4.3]{z79}}. If $ a\in E $, then the following are equivalent:
\begin{enumerate}
	\item[(i)] $ a $ is central;
	\item[(ii)] $ aE\subseteq E $;
	\item[(iii)] $ a\orb(a) \subseteq \orb(a) $;
	\item[(iv)] $ \displaystyle\sup_{r\in \orb(a)} \rho(ar)<\infty$;
	\item[(v)] $ \displaystyle\sup_{r\in \orb(a)} \rho(r+a)<\infty$;
	\item[(vi)] $ \displaystyle\sup_{r\in \orb(a)} \rho(r-a)<\infty$;
	\item[(vii)] $ \displaystyle\sup_{r\in \orb(a)} \rho(a\circ r)<\infty$;
	\item[(viii)] $ a\circ E\subseteq E $;
	\item[(ix)] $ a\circ \orb(a)\subseteq \orb(a) $.
	
\end{enumerate}


In the present paper, we not only show this list can be extended to members of $ E(p) $, but also a portion thereof is true for arbitrary Banach algebra elements. Furthermore, we had sought to obtain new information concerning the interactions between the elements belonging to different similarity orbits, say $\orb(a)$ and $\orb(b)$, in $A$. This is in stark contrast to the results in {\cite{z79}} where members of $ \orb(a) $ and only their interactions with $ a $ were studied. To clarify, compare the next question with (iv) in the above list: Suppose there is some $ \epsilon>0 $ such that $ \rho(ar)\leq \epsilon $ for all $ r\in \orb(b) $; what is the connection between $ \orb(a) $ and $ \orb(b) $? We answer this question via Corollary \ref{artificialList}, which says that, under the above premise, the similarity orbits must commute (with the converse also being true). Keep in mind that there need not be anything extra assumed for the elements; they need not necessarily be idempotents or even algebraic. Moreover, we also consider other algebraic variations, namely, $ r+a $, $ r-a $, and $ a\circ r $, as $ r $ varies through $ \orb(b) $—and the same conclusions relating $ \orb(a) $ and $ \orb(b) $ are deduced.

The improvements to the results due to Zemánek which we have in our current paper are corollaries to a more general theorem—namely—Theorem \ref{Musing}. In order to guarantee the commutation of the similarity orbits of $ a $ and $ b $, we need only require that the union of all $ \sigma^\prime(ar) $—as $ r $ runs through $ \orb(b) $—misses at least one nonzero complex number. We obtain the same statement for $ r+a $ and $ r-a $ and an ever-so-slightly weaker statement for $ a\circ r $. The details are thoroughly discussed in the theorem together with its proof. This is the highlighting result of Section \ref{sec1}. Rounding off this section is Theorem \ref{classical}, wherein we obtain an improvement upon two classical characterizations of commutative semisimple Banach algebras related to subadditivity and submultiplicativity (respectively) of the spectral radius. The precise details of this will be expounded upon thoroughly prior to Theorem \ref{classical}.

In Section \ref{sec3}, a weak localization of Theorem \ref{Musing} is proven. Instead of $ r $ varying throughout $ \orb(b) $, we consider when $ r $ varies in some open subset $ N $ of $ \orb(b) $—but—we must add a few additional conditions in the premise. The first being that $ \sigma(ar)\subseteq C $, for all $r\in N$ where $ C $ is a fixed countable subset of $ \field $; the second assumption is that $ ar $ has a finite spectrum for some $ r\in\orb(b) $. Then we can conclude that the similarity orbits of $ a $ and $ b $ commute. 
 These additional assumptions are actually quite natural when one considers algebraic elements, as we do in the next section—wherein we apply the theory resulting from the prior sections in Section \ref{sec4} to algebraic elements, and then to idempotents, so that new facts concerning the similarity orbits of these types of elements are brought to light.
 
In this final section, supposing that $ a $ and $ b $ belong to $E(p) $, we expose some conditions under which $ a^\prime b^\prime=\mu\identity $, for all $ (a^\prime,b^\prime)\in\orb(a)\times\orb(b) $—and $ \mu\in\field $ is fixed. Then, appropriately, if $ a,b\in E$, it turns out that $ \mu=0 $, so that the similarity orbits of $ a $ and $ b $ are orthogonal. The pair of concluding results in this section discusses some localization aspects regarding idempotents and algebraic elements in a similar vein to the main theorem of Section \ref{sec3}.

Much of the early work in this paper deals with representation theory for Banach algebras. For a brief idea of a typical argument using this, one can look at the proof of {\cite[Theorem 2.3]{radicalpaper}} wherein a characterization of the Jacobson radical, $ \rad(A) $, is given. A Banach algebra $ A $ is said to be semisimple when $ \rad(A)=\{0\} $. We shall use the fact that $ \rad(A) $ is precisely those members of $ A $ which belong to the kernels of all continuous irreducible representations $ \pi $ of $ A $ throughout our work (see {\cite[Theorem 4.2.1]{a91}}). Along with that, the next result due to A. Sinclair will be crucial to our proofs:
	\begin{theorem}[{Sinclair Density Theorem, \cite[Corollary 4.2.6]{a91}}]\label{SDT}
		Let $ \pi $ be a continuous irreducible representation of a Banach algebra $ A $ on a Banach space $ X $. If $ \left\{\xi_1,\ldots,\xi_n\right\} $ and $ \left\{\eta_1,\ldots,\eta_n\right\} $ are linearly independent subsets of $ X $, then there exists $ y\in G_\identity (A)$ such that 
		$$\pi(y)\xi_i=\eta_i \text{ for each } i\in\{1,\ldots,n\}.$$
	\end{theorem}	
	 We set forth a few additional conventions which will be the standard for the rest of this paper. Let $  L,M \subseteq \field $. To denote the number of distinct elements in $ L $, we will simply write $ \#L $, and we also define the following sets:
\begin{align*}
	LM\coloneqq&\left\{\lambda\mu\colon \lambda\in L,\mu\in M\right\}, \\
	L+M\coloneqq&\left\{\lambda+\mu\colon\lambda\in L,\mu\in M\right\}, \\
	\alpha L\coloneqq& \left\{\alpha\lambda\colon \lambda\in L\right\} \text{ for } \alpha\in\field.
\end{align*}
For the sake of brevity, we will let $ L^2 $ denote the set $ LM $ whenever $ L=M $. If $ a,b\in A $ commute, then the following spectral containments concerning their sum and product are true:
$$\sigma(a+b)\subseteq \sigma(a)+\sigma(b), \quad \sigma(ab)\subseteq \sigma(a)\sigma(b).$$ 
For any metric space $ \mathcal{M} $, we will denote the open and closed balls with centre $ x\in \mathcal{M} $ and radius $ \epsilon $ by $ B_\mathcal{M}(x,\epsilon) $ and $ \overline{B}_\mathcal{M}(x,\epsilon)  $, respectively. 
	\section{Commutativity of Similarity Orbits}\label{sec1}

The lemma below is the backbone of the proof of the main result—Theorem \ref{Musing}—in that most of the nuances regarding the representation theory used lie herein. The proof hinges on the fact that, if $ X $ is a Banach space and $ T\in\mathcal{L}(X) $, the Banach algebra of bounded linear operators on $ X $, then $ T $ is a scalar multiple of the identity operator whenever every nonzero $ \xi\in X $ is an eigenvector of $ T $. Hence if $ T $ is non-central we must be able to find $ \xi\in X $ such that $ \{T\xi, \xi\} $ is linearly independent in $ X $.
	\begin{lemma}\label{theplotthickens}
		Let $ \pi $ be a continuous irreducible representation of a Banach algebra $ A $ on a Banach space $ X $. Suppose that $ \pi(a) $ is non-central in $ \mathcal{L}(X) $ and $ \alpha\in\field \setminus\{0\}$. Then there exist $ \eta\in X\setminus\{0\} $ and $ \{x,y,z\}\subseteq G_\identity(A) $ such that $ \{\eta,\pi(a)\eta\} $ is a linearly independent set and
		\begin{align*}
			\pi(x)^{-1}\pi(a)\pi(x)\pi(a)\eta=&\alpha\eta, \\
			\left(\pi(y)^{-1}\pi(a)\pi(y)+\pi(a)\right)\eta=&\alpha\eta, \\
			\left(\pi(z)^{-1}\pi(a)\pi(z)-\pi(a)\right)\eta=&\alpha\eta.
		\end{align*}
		Moreover, there exist $ \gamma\in X\setminus\{0\} $ and $ w\in G_\identity(A) $ such that $ \{\gamma,\pi(a)\gamma\} $ is linearly independent and $\pi\left((w^{-1}aw)\circ a\right)\gamma = (1-\alpha)\gamma.  $
	\end{lemma}
	\begin{proof}
		If $ \{\pi(a)\eta,\eta\} $ is a linearly dependent set in $ X $ for all $ \eta\in X $, then it can be shown that $ \pi(a) $ is a scalar multiple of the identity operator, implying that $ \pi(a) $ is central—which is obviously a contradiction. So let $ \eta\neq 0 $ be such that $ \{\pi(a)\eta,\eta\} $ is a linearly independent set. We can certainly find some $ \lambda\in\field $ such that $ \lambda^2=\alpha $, and then $ \{\pi(a)\eta,\lambda\eta\} $ is also linearly independent. By Sinclair's Density Theorem we obtain some $x\in G_\identity(A) $ such that
		\begin{align*}
				\pi(x)\pi(a)\eta=&\lambda\eta, \\
				\pi(x)(\lambda\eta)=&\pi(a)\eta.
			\end{align*}
	Rewriting the second equality here gives $ \lambda\eta=\pi(x)^{-1}\pi(a)\eta $. Therefore, 
	$$\pi\left(x^{-1}axa\right)\eta=\pi\left(x^{-1}a\right)(\lambda\eta)=\lambda^2\eta=\alpha\eta.$$
	One can easily show that $ \{\beta\pi(a)\eta-\beta\alpha\eta,\eta\} $ is linearly independent for all complex numbers $ \beta\neq 0 $. Then recall $ \{\pi(a)\eta,\eta\} $ is linearly independent, so that an application of Sinclair's Density Theorem yields some $ y_\beta\in G_\identity(A) $ such that
	\begin{align*}
		\pi\left(y_\beta\right)\pi(a)\eta =&\beta\pi(a)\eta-\beta\alpha\eta \\
		\pi\left(y_\beta\right)\eta=&\eta.
	\end{align*}
	With the second equality above we have that
	$$\pi(a)\pi\left(y_\beta\right)\eta=\pi(a)\eta,$$
	and then the following is true:
	\begin{align*}
		\pi\left(y_\beta\right)^{-1}\left[\pi\left(y_\beta\right)\pi(a)+\pi(a)\pi\left(y_\beta\right)\right]\eta =&\pi\left(y_\beta\right)^{-1} \left[(\beta+1)\pi(a)\eta-\beta\alpha\eta\right], \\
		\pi\left(y_\beta\right)^{-1}\left[\pi(a)\pi\left(y_\beta\right)-\pi\left(y_\beta\right)\pi(a)\right]\eta =&\pi\left(y_\beta\right)^{-1} \left[(1-\beta)\pi(a)\eta+\beta\alpha\eta\right].
	\end{align*}
Since $\pi\left(y_\beta\right)^{-1}\eta=\eta$, taking $ y=y_\beta$ when $\beta=-1 $ gives us $$\left(\pi(y)^{-1}\pi(a)\pi(y)+\pi(a)\right)\eta=\alpha\eta.$$
On the other hand if we set $ z=y_\beta $ when $ \beta=1 $ we obtain
$$\left(\pi(z)^{-1}\pi(a)\pi(z)-\pi(a)\right)\eta=\alpha\eta.$$
Finally, if we take $ a^\prime=\identity-a $, then $ \pi(a^\prime) $ is non-central, and hence, by the first part, there is $ w\in G_\identity(A) $ and $ \gamma\in X\setminus\{0\} $ such that $ \{\pi(a^\prime)\gamma,\gamma\} $ is linearly independent and 
$$\pi(w^{-1}a^\prime w a^\prime)\gamma = \alpha\gamma.$$
Hence it follows that $ \{\pi(a)\gamma,\gamma\} $ is linearly independent and
\begin{align*}
	\pi((w^{-1}aw)\circ a)\gamma =& \pi(\identity-(\identity-w^{-1}aw)(\identity-a))\gamma \\
	 =& \pi(\identity-w^{-1}a^\prime w a^\prime)\gamma = \gamma-\alpha\gamma=(1-\alpha)\gamma.
\end{align*}
	\end{proof}
Since most of the technicalities were taken care of in the above lemma, the proof of Theorem \ref{Musing} is straightforward. The key point here is that the eigenvalues of an operator $ \pi(a) $ (for any continuous irreducible representation $ \pi $ of $ A $) belong to $ \sigma_A(a) $. Considering the hypothesis of the lemma once more, we observe that we have near-total freedom in our choice of the complex number $ \alpha $ (the only restriction being that it is nonzero, with 1 also being excluded for the form $ a\circ r $). From this, conditions (\ref{(i)})--(\ref{(iv)}) become obvious.
\begin{theorem}\label{Musing}
	Let $ A $ be a semisimple Banach algebra and suppose that $ a,b\in A $. Then every element of $ \orb(a) $ commutes with every element of $ \orb(b) $ if and only if any one of the following holds:
	\begin{align}
		\label{(i)}
		&\bigcup_{r\in \orb(b)}\sigma^\prime(ar)\neq \field\setminus\{0\} \\ \label{(ii)}
		&\bigcup_{r\in \orb(b)}\sigma^\prime(r+a)\neq \field\setminus\{0\} \\ \label{(iii)}
		&\bigcup_{r\in \orb(b)}\sigma^\prime(r-a)\neq \field\setminus\{0\}  \\ \label{(iv)}
		&\bigcup_{r\in \orb(b)}\left(\sigma^\prime(a\circ r)\setminus\{1\}\right)\neq \field\setminus\{0,1\}. 
	\end{align}
\end{theorem}
\begin{proof}
	For the forward implication, we may use commutation of $ r $ and $ a $ to obtain
	\begin{align*}
		\sigma(ar) &\subseteq \sigma(a)\sigma(r),&\quad
		\sigma(r+a) &\subseteq \sigma(a)+\sigma(r),\\
		\sigma(r-a) &\subseteq \sigma(r)-\sigma(a),&\quad
		\sigma(a\circ r) &\subseteq \sigma(a)+\sigma(r)-\sigma(a)\sigma(r).
	\end{align*}
	Recall that $\sigma(r)=\sigma(b)$ for all $ r\in \orb(b) $, and hence our claim is true. With regards to the converse, we will show commutation of $ a $ and $ b $ under the assumption of (\ref{(i)})-(\ref{(iv)}) individually. So suppose that $ ab\neq ba $, and let (\ref{(i)}) hold. By semisimplicity there is a continuous irreducible representation $ \pi $ of $ A $ on some Banach space $ X $ such that $ \pi(a) $ and $ \pi(b) $ are not central. Let $ \alpha\in\field\setminus\{0\}$ such that $ \alpha\not\in \sigma(ar) $ for all $ r\in \orb(b) $. Apply Lemma \ref{theplotthickens} to obtain some $ \eta\neq 0 $ and $ y\in G_\identity(A) $ such that $ \{\pi(a)\eta,\eta\} $ is linearly independent and 
	$$\pi(y)^{-1}\pi(a)\pi(y)\pi(a)\eta=\alpha\eta.$$
	Moreover, fix $\gamma\neq 0$ such that $ \left\{\pi(b)\gamma,\gamma\right\} $ is linearly independent, so that with Sinclair's Density Theorem there must exist $ w\in G_\identity(A) $ such that 
	\begin{align*}
		\pi(w)\pi(a)\eta =& \pi(b)\gamma, \\
		\pi(w)\eta =& \gamma.
	\end{align*}
	Therefore, one has that
	\begin{align*}
		\pi\left(y^{-1}ayw^{-1}bw\right)\eta = \pi\left(y^{-1}ay\right) \pi\left(w^{-1}b\right)\gamma = \pi\left(y^{-1}ay\right)\pi(a)\eta = \alpha\eta.
	\end{align*}
	Notice that this together with Jacobson's Lemma forces $ \alpha\in\sigma\left(ayw^{-1}bwy^{-1}\right) $ which is a contradiction to our choice of $ \alpha $. Thus, $ a $ and $ b $ commute. \\
	Now assume that (\ref{(ii)}) holds but $ ab\neq ba $ and let $ \alpha\in\field\setminus\{0\}$ such that $ \alpha\not\in \sigma(a+r) $ for all $ r\in \orb(b) $. Once more, there is a continuous irreducible representation $ \pi $ on $ A $ for some Banach space $ X $ such that $ \pi(a) $ and $ \pi(b) $ are not central. With Lemma \ref{theplotthickens} we obtain some $ \eta\neq 0 $ and $ y\in G_\identity(A) $ such that $ \{\pi(a)\eta,\eta\} $ is linearly independent and 
$$
		\left(\pi(y)^{-1}\pi(a)\pi(y)+\pi(a)\right)\eta=\alpha\eta.
$$
	Also fix $ \gamma\neq 0 $ such that $ \{\pi(b)\gamma,\gamma\} $ is a linearly independent set. By Sinclair's Density Theorem there exists $ w\in G_\identity(A) $ such that
	\begin{align*}
		\pi(w)\pi(a)\eta =& \pi(b)\gamma, \\
		\pi(w)\eta =& \gamma.
	\end{align*}
	Hence it follows that
	\begin{align*}
		\pi\left(y^{-1}ay+w^{-1}bw\right)\eta =& \pi\left(y^{-1}ay\right)\eta+\pi\left(w^{-1}bw\right)\eta  = \pi\left(y^{-1}ay\right)\eta +\pi\left(w^{-1}b\right)\gamma \\=& \pi\left(y^{-1}ay\right)\eta +\pi(a) \eta = \alpha\eta.
	\end{align*}
	Notice that this implies $ \alpha\in\sigma\left(a+yw^{-1}bwy^{-1}\right) $ which is a contradiction to our choice of $ \alpha $. Thus, $ a $ and $ b $ commute. \\
	The case for (\ref{(iii)}) is analogous to the one above. So assume that (\ref{(iv)}) holds but $ ab\neq ba $. Then $ \pi(a) $ and $ \pi(b) $ are non-central for a continuous irreducible representation $ \pi $ of $ A $ on a Banach space $ X $. Pick $ \alpha^\prime\not\in \field\setminus\{0,1\} $, so that $ 1-\alpha^\prime\neq 0 $. Therefore, Lemma \ref{theplotthickens} says that there exist $ \zeta\in X\setminus\{0\} $ and $ z\in G_\identity(A) $ such that $ \{\pi(a)\zeta,\zeta\} $ is linearly independent and
		$$\pi((z^{-1}az)\circ a)\zeta = \alpha^\prime \zeta.$$
	Moreover, let $ \omega\in X\setminus\{0\} $ be chosen so that $ \{\pi(b)\omega,\omega\} $ is linearly independent in $ X $; hence, Sinclair's Density Theorem gives rise to $ u\in G_\identity(A) $ such that 
	\begin{align*}
		\pi(u)\pi(a)\zeta =& \pi(b)\omega, \\
		\pi(u)\zeta =& \omega.
	\end{align*}
	Now observe that we have the following:
	\begin{align*}
		\pi((zaz^{-1})\circ(u^{-1}bu))\zeta =& \pi(zaz^{-1})\zeta+\pi(u^{-1}bu)\zeta-\pi(zaz^{-1}u^{-1}bu)\zeta \\=& \pi(zaz^{-1})\zeta + \pi(a)\zeta-\pi(zaz^{-1})\pi(a)\zeta 
		\\=& \pi((zaz^{-1})\circ a)\zeta = \alpha^\prime \zeta.
	\end{align*}
	Thus, Jacobson's Lemma says that $ \alpha^\prime\in \sigma(a\circ(z^{-1}u^{-1}buz)) $. But because $ \alpha^\prime $ was arbitrary in $ \field\setminus\{0,1\} $, we now have a contradiction to (\ref{(iv)}). So $ a $ and $ b $ do commute. \\
	Finally, observe that, for all the above arguments, we may use Jacobson's Lemma to replace $ a $ and $ b $ by arbitrary members coming from $ \orb(a) $ and $ \orb(b) $ respectively, so that the proof is now complete.
\end{proof}

Under commutation of the similarity orbits, it follows from the commutation of $a $ and $ r $ (for each $ r\in\orb(b) $) that $ \sigma(ar) $, $ \sigma(r+a) $, $ \sigma(r-a) $, and $ \sigma(a\circ r) $ are all contained in respective fixed compact subsets of $ \field $ (specific to each of the four forms). The precise details of this is actually used in the proof below. Now by replacing $ \sigma $ here by $ \hat{\sigma} $ (the polynomially convex hull of the spectrum), we get not just set containment—but equality to some fixed compact subsets of the complex plane.
\begin{corollary}
	Let $ A $ be a semisimple Banach algebra and suppose that $ a,b\in A $. If any one of the conditions \textnormal{(\ref{(i)})}-\textnormal{(\ref{(iv)})} holds, then for each $ r\in \orb(b) $ we have that
	\begin{align*}
		\hat{\sigma}(ar)&=\hat{\sigma}(ab), \quad& \hat{\sigma}(r+a)&=\hat{\sigma}(b+a),\\  \hat{\sigma}(r-a)&=\hat{\sigma}(b-a), \quad& \hat{\sigma}(a\circ r)&=\hat{\sigma}(a \circ b).
	\end{align*}
\end{corollary}
\begin{proof}
	If any of \textnormal{(\ref{(i)})}-\textnormal{(\ref{(iv)})} is true, then by Theorem \ref{Musing} it follows that $ a $ commutes with any $ r\in\orb(b) $. Suppose that $ r=e^{x_1}\cdots e^{x_n}be^{-x_n}\cdots e^{-x_1} $ is a fixed member of $ \orb(b) $ and define the entire function $ f\colon \field\to \orb(b) $ as follows:
	$$f(\lambda)\coloneqq e^{\lambda x_1}\cdots e^{\lambda x_n}be^{-\lambda x_n}\cdots e^{- \lambda x_1}.$$
	Since $ a $ commutes with $ f(\lambda) $ and $ \sigma(f(\lambda)) =\sigma(b) $, for all $ \lambda\in \field $, we have the following:
	\begin{align*}
		\sigma(f(\lambda)+a)&\subseteq \sigma(b)+\sigma(a), \quad&
		\sigma(f(\lambda)-a)&\subseteq \sigma(b)-\sigma(a), \\
			\sigma(a\circ f(\lambda))&\subseteq \sigma(a)+\sigma(b)-\sigma(a)\sigma(b) \quad& \sigma(af(\lambda))&\subseteq\sigma(a)\sigma(b).
	\end{align*}
	Hence Liouville's Spectral Theorem ({\cite[Theorem 3.4.14]{a91}}) implies that $ \hat{\sigma}(af(\lambda)) $, $ \hat{\sigma}(f(\lambda)+a) $, $ \hat{\sigma}(f(\lambda)-a) $, $ \hat{\sigma}(a\circ f(\lambda)) $ are all constant, so we have that
	\begin{align*}
		\hat{\sigma}(ar)&=\hat{\sigma}(af(1))=\hat{\sigma}(af(0))=\hat{\sigma}(ab), \\\hat{\sigma}(r+a)&=\hat{\sigma}(f(1)+a)=\hat{\sigma}(f(0)+a)=\hat{\sigma}(b+a),\\  \hat{\sigma}(r-a)&=\hat{\sigma}(f(1)-a)=\hat{\sigma}(f(0)-a)=\hat{\sigma}(b-a), \\ \hat{\sigma}(a\circ r)&=\hat{\sigma}(a\circ f(1))=\hat{\sigma}(a\circ f(0))=\hat{\sigma}(a \circ b).
	\end{align*}
	Since $ r $ was arbitrarily chosen from $ \orb(b) $, we now have the result.
\end{proof}
Our next corollary considers the case where $ a $ and $ b $ have the same similarity orbit. It turns out that $ a $ commutes with every member of $ A $ when we have that $ \orb(a)=\orb(b) $ and if one of (\ref{(i)})-(\ref{(iv)}) holds. 
\begin{corollary}\label{MusingPt2}
	Let $ A $ be a semisimple Banach algebra and suppose that $ a\in A $. Then $ a $ is central if and only if any one of the following holds:
	\begin{align}
		\label{(v)}
		&\bigcup_{r\in \orb(a)}\sigma^\prime(ar)\neq \field\setminus\{0\} \\ \label{(vi)}
		&\bigcup_{r\in \orb(a)}\sigma^\prime(r+a)\neq \field\setminus\{0\} \\ \label{(vii)}
		&\bigcup_{r\in \orb(a)}\sigma^\prime(r-a)\neq \field\setminus\{0\}  \\ \label{(viii)}
		&\bigcup_{r\in \orb(a)}\left(\sigma^\prime(a\circ r)\setminus\{1\}\right)\neq \field\setminus\{0,1\}. 
	\end{align}
\end{corollary}
\begin{proof}
	Clearly, when $ a $ is central, we will have that $ \orb(a)=\{a\} $, so then (\ref{(v)})-(\ref{(viii)}) are all true. Conversely, if any of the conditions (\ref{(v)})-(\ref{(viii)}) holds, then by Theorem \ref{Musing} it follows that $ a $ commutes with every member of $ \orb(a) $. If, to the contrary, $ a $ is not central, then there is a continuous irreducible representation $ \pi $ on a Banach space $ X $ such that $ \pi(a) $ is non-central. Pick $ \alpha\in \field\setminus\left(\{0\}\cup \sigma(a)^2\right) $. By Lemma \ref{theplotthickens} there exist $ w\in \orb(a) $ and $ 0\neq\eta\in X $ such that $$\pi(w^{-1}awa)\eta=\alpha\eta.$$
	Hence, $ \alpha\in \sigma(w^{-1}awa) $. But because $ a $ commutes with $ w^{-1}aw $, we may conclude that $ \sigma(w^{-1}awa) \subseteq \sigma(w^{-1}aw)\sigma(a)=\sigma(a)^2  $, and therefore $ \alpha\in\sigma(a)^2 $—which is clearly a contradiction. Therefore, $ a $ is central.
\end{proof}
The following corollary, which is immediate from Theorem \ref{Musing} and Corollary \ref{MusingPt2}, generalizes some of the statements given in {\cite[Theorem 4.3]{z79}} to general Banach algebra elements:
	\begin{corollary}\label{artificialList}
			Let $ A $ be a semisimple Banach algebra and suppose that $ a,b\in A $. Then every element of $ \orb(a) $ commutes with every element of $ \orb(b) $ if and only if any one of the following holds:
			\begin{enumerate}
				\item[\textnormal{(i)}] $\displaystyle \sup_{r\in \orb(b)}\rho(ar) <\infty$;
				\item[\textnormal{(ii)}] $ \displaystyle\sup_{r\in \orb(b)}\rho(r+a) <\infty$;
				\item[\textnormal{(iii)}] $ \displaystyle\sup_{r\in \orb(b)}\rho(r-a) <\infty$;
				\item[\textnormal{(iv)}] $\displaystyle \sup_{r\in \orb(b)}\rho(a\circ r) <\infty$.
			\end{enumerate}
		In particular, if $ \orb(a)=\orb(b) $, then $ a $ is central if and only if any one of the conditions \textnormal{(i)-(iv)} holds.
	\end{corollary}
	The above corollary can be used to strengthen two well-known characterizations of commutative semisimple Banach algebras. It is known (see for instance \cite[Corollary 5.2.3]{a91} or the main theorem in \cite{Zemanek}) that a semisimple Banach algebra $ A $ is commutative if and only if any one of the following conditions hold:
	\begin{enumerate}
		\item[(i)] there exists $ M>0 $ such that $ \rho(a+b)\leq M(\rho(a)+\rho(b)) $ for all $ a,b\in A $;
		\item[(ii)]  there exists $ N>0 $ such that $ \rho(ab)\leq N\rho(a)\rho(b) $ for all $ a,b\in A $.
	\end{enumerate}
	Obviously, if $ A $ is commutative, then we can take $ M=N=1 $. Observe that both $ M $ and $ N $ are independent of all members $ a,b\in A $, so the conditions are required to hold globally on $A$. The following theorem localizes these results to similarity orbits—which, as a reminder, partition $ A $.
	\begin{theorem}\label{classical}
		Let $ A $ be a semisimple Banach algebra. Then the following properties are equivalent:
		\begin{enumerate}
			\item[\textnormal{(i)}] $ A $ is commutative;
			\item[\textnormal{(ii)}] for each similarity orbit $ \orb(a) $ of $ A $, there exist $ M_{\orb(a)}>0 $ and $ a^\prime\in \orb(a) $ such that $ \rho(a^\prime+r)\leq M_{\orb(a)}(\rho(a^\prime)+\rho(r)) $ for all $ r\in \orb(a) $;
			\item[\textnormal{(iii)}] for each similarity orbit $ \orb(a) $ of $ A $, there exist $ N_{\orb(a)}>0 $ and $ a^\prime\in \orb(a) $ such that $ \rho(a^\prime r)\leq N_{\orb(a)}\rho(a^\prime)\rho(r) $ for all $ r\in \orb(a) $.
		\end{enumerate}
	\end{theorem}
	\begin{proof}
		If $ A $ is commutative then (ii) and (iii) are trivially true. So suppose that (ii) holds. Then, for each $ r\in\orb(a) $, it follows that
		$$ \rho(a^\prime+r)\leq M_{\orb(a)}(\rho(a^\prime)+\rho(r)) =2M_{\orb(a)} \rho(a),$$
		whence Corollary \ref{artificialList} implies that $ a^\prime $ is central, and so, $ a^\prime=a $. Since $ a\in A $ was arbitrary, it follows that $ A $ is commutative. If (iii) holds then the argument is analogous.
 	\end{proof}
	\section{Localization to Neighbourhoods in Similarity Orbits}\label{sec3}
	If $ K_1 $ and $ K_2 $ are nonempty compact subsets of $ \field $, then the Hausdorff distance between these sets is defined as:
	$$\Delta(K_1,K_2)\coloneqq \max\left\{\sup_{\lambda\in K_1}\left(\inf_{\alpha\in K_2}\abs{\lambda-\alpha}\right),\sup_{\lambda\in K_2}\left(\inf_{\alpha\in K_1}\abs{\lambda-\alpha}\right)\right\}.$$ 
	The main result of this section is Theorem \ref{mainLocalization} and, for the sake of brevity, we state it for only one type of operation, namely, multiplication. That is, the hypothesis and conclusion to the theorem stated below pertain only to the spectra of elements of the form $ ra $ as $ r $ varies through $ \orb(b) $.	The reader will no doubt find after perusing the proof of the theorem that there is really no need to limit the result to products of elements, but we may obtain identical results for each of the remaining combinations of elements: $ r+a $, $ r-a $ and $ r\circ a $. Refer to the remarks at the end of this section to view the tweaks necessary to achieve this.
	\begin{theorem}\label{mainLocalization}
		Let $ A $ be a semisimple Banach algebra. Suppose that $ a,b\in A $ and $ C $ is a fixed countable subset of $ \field $. If there exists a nonempty open subset $ N $ of $ \orb(b) $ and a pair $ (a^\prime,b^\prime)\in \orb(a)\times\orb(b) $ such that $ \sigma(ra)\subseteq C $ for all $r\in N $ and $ \#\sigma( b^\prime a^\prime)<\infty $, then every member of $ \orb(a) $ commutes with every member of $ \orb (b)$. Moreover, we have that $ \sigma(ra)=\sigma(ba) $ for all $ r\in\orb(b) $.
	\end{theorem}
	The following two lemmas are a requisite for proving Theorem \ref{mainLocalization}, and the hypothesis thereof will be assumed in the premise of each of these lemmas.
	\begin{lemma}\label{countable}
		$ \sigma(ra) $ is countable for all $ r\in \orb(b) $.
	\end{lemma}
	\begin{proof}
		Because the map $ v\mapsto vbv^{-1}$ is continuous on $ G_\identity $, it follows that there exists an open ball $ B=B_A(v_0,\epsilon) $ in $ G_\identity $ such that $ v\in B $ implies $ vbv^{-1} \in N $ and hence $ \sigma(vbv^{-1}a)\subseteq C $ for all $ v\in B $. Suppose that $ v_0=e^{x_1}\cdots e^{x_n} $. Now take $ y=e^{y_1}\cdots e^{y_m}\in G_\identity $ to be arbitrary. By appropriately using $e^0=\identity$, we may assume without loss of generality that $m=n$. Next we define the entire function $ f\colon \field\to G_\identity $ to be
		$$f(\lambda)\coloneqq e^{(1-\lambda)x_1+\lambda y_1}\cdots e^{(1-\lambda)x_n+\lambda y_n}.$$
		Then if $ g(\lambda)=f(\lambda)b\left[f(\lambda)\right]^{-1}a $, we have via \cite[Theorem 7.1.13]{a91} that $ \lambda\mapsto \sigma(g(\lambda)) $ is an analytic multifunction on $ \field $. Furthermore, the continuity of $ f $ says that there exists $ \delta>0 $ such that $ \abs{\lambda}<\delta $ implies $ f(\lambda)\in B $, whence $ \sigma(g(\lambda)) \subseteq C$ whenever $ \abs{\lambda}<\delta $. But then the set of $ \lambda\in\field $ such that $ \sigma(g(\lambda)) $ is at most countable has nonzero capacity—in which case, \cite[Theorem 7.2.8]{a91} forces $ \sigma(g(\lambda)) $ to be at most countable for all $ \lambda\in \field $. In particular then, $ \sigma(g(1)) =\sigma(yby^{-1}a)$ is at most countable. Because $ y\in G_\identity $ was arbitrary, we now have the result.
	\end{proof}
	\begin{lemma}\label{TwoProofs}
		There exists an open ball $ B $ in $ G_\identity $ such that $ \sigma(vbv^{-1}a) $ is constant for all $ v\in B $.
	\end{lemma}
	\begin{proof}
		Denote by $ \mathcal{K} $ the metric space consisting of all nonempty compact subsets of $ \field $ under the Hausdorff metric. Let $ T\colon G_\identity\to\mathcal{K} $ be the map $ v\mapsto\sigma(vbv^{-1}a) $. The open ball $ B $ will be the same one which we obtained in Lemma \ref{countable}, that is, $ B=B_A(v_0,\epsilon) $. By continuity of $ T $ and the connectedness of $ B $, it follows that $ T(B) $ is connected in $ \mathcal{K} $. To complete the proof, we need only show that $ T(B) $ is a singleton. To this end, we will assume to the contrary that there is some $ v\in B $ such that $ T(v)\neq T(v_0) $. Now there is an open ball $ D_0= B_{T(B)}(T(v_0),\delta_0) $ in $ T(B) $ to which $ T(v) $ does not belong. Moreover, if there does not exist any $ T(y) $ on the boundary $ \partial D_0 $ of the open ball then that would mean that $ T(B) $ is the union of two nonempty disjoint open sets, that is, $$ T(B)=D_0\cup (T(B)\setminus\overline{D_0}), $$ which is obviously a contradiction. So let $ T(y_0) \in \partial D_0 $. Therefore, we can use the same reasoning to guarantee that for each real number $ \delta\in \left(0,\delta_0\right]$, there is some $ y_\delta\in B $ such that $ T(y_\delta)\in \partial D_\delta $, where $ D_\delta=B_{T(B)}(T(v_0),\delta) $ here. Recall that by the hypothesis, we have that $ T(v)\subseteq C $ for all $ v\in B $, where $ C $ is a fixed countable subset of $ \field $. Because the spectrum is nonempty and compact, we have that if $ T(y_\delta) \in \partial D_\delta$ then the distance of $ T(y_\delta) $ to $ T(v_0) $ is assumed at some complex numbers belonging to $ C $. Therefore, to each $ \delta\in(0,\delta_0] $ we can associated a pair $ (\alpha_\delta,\beta_\delta) \in T(y_\delta)\times T(v_0)$ such that $ \abs{\alpha_\delta-\beta_\delta}=\delta $. Clearly, $ C\times C $ is countable whilst the interval $ (0,\delta_0] $ is uncountable, and so the preceding statement leads to the following pigeonhole argument: There must be distinct real numbers, say $ \delta $ and $ \gamma $ in $ (0,\delta_0] $ for which the corresponding pairs $ (\alpha_\delta,\beta_\delta), (\alpha_\gamma,\beta_\gamma) \in T(y_\delta)\times T(v_0)$ are the same. That is, $ \alpha_\delta=\alpha_\gamma $ and $ \beta_\delta=\beta_\gamma $ in the above, in addition to the fact that
		$$\abs{\alpha_\delta-\beta_\delta}=\delta, \quad \abs{\alpha_\gamma-\beta_\gamma}=\gamma.$$
		Notice then that the contradiction $ \delta=\gamma $ is reached.
	\end{proof}

 	\begin{proof}[Proof of Theorem \ref{mainLocalization}]
 		In Lemma \ref{TwoProofs} we established that $ \sigma(vbv^{-1}a) $ is constant for all $ v\in B $, where $ B=B_A(v_0,\epsilon) $ is an open ball in $ G_\identity $. Let us assume that, for each $ v\in B $, $$ \sigma(vbv^{-1}a)=\left\{\alpha_n\right\}_{n\in\mathbb{N}}=M, $$ where $ M\subseteq C $. Suppose that $ y\in G_\identity $ is arbitrary, and define the functions $ f $ and $ g $ as in Lemma \ref{countable}, with $ v_0=e^{x_1}\cdots e^{x_n} $ and $ y=e^{y_1}\cdots e^{y_n} $. Then the map $ \lambda\mapsto \sigma(g(\lambda)) $ is an analytic multifunction; with Lemma \ref{countable} we have that $  \sigma(g(\lambda)) $ is countable for all $ \lambda\in \field $ and by our preceding remarks, there exists $ \delta>0 $ such that $ \abs{\lambda} <\delta$ implies $ \sigma(g(\lambda)) =M $. Now for each $ n\in\mathbb{N} $ we define the set $$Z_n\coloneqq\left\{\lambda\in\field\colon \alpha_n\in  \sigma(g(\lambda)) \right\},$$
 		and we observe that $ B_\field(0,\delta)\subseteq Z_n $ for all $ n\in\mathbb{N} $—in which case, \cite[Theorem 7.2.13]{a91} says that $ Z_n=\field $ for each $ n\in\mathbb{N} $. In particular then, $ \alpha_n\in  \sigma(g(1))  =\sigma(yby^{-1}a)$ for all $ n\in\mathbb{N} $. As $ y $ was arbitrarily picked in $ G_\identity $ we can conclude that $ M\subseteq \sigma(yby^{-1}a) $ for all $ y\in G_\identity $. \\ 
 		By the hypothesis and Jacobson's Lemma, it follows that there is some $ y_0\in G_\identity $ such that $ \sigma(y_0by_0^{-1}a) $ is finite, and consequently, $ M $ is finite. But recall we had that $ \sigma(vbv^{-1}a)=M $ for all $ v\in B $. So we may fix an arbitrary $ y\in G_\identity $ as before and define the corresponding analytic function $ g $ according to Lemma \ref{countable} to obtain existence of a $ \delta>0 $ (via continuity of $ g $) such that $ \abs{\lambda}<\delta $ implies that $\# \sigma(g(\lambda)) <\infty$. Hence the Scarcity Theorem (refer to {\cite[Theorem 3.4.25]{a91}}) now says that $ \#\sigma(g(\lambda)) \leq \#M$ for all $ \lambda\in\field $. But by the first part, we have that $ \sigma(g(\lambda)) $ contains $ M $ for all $ \lambda\in\field $ and, in particular, $M\subseteq \sigma(g(1)) $, so it follows that $ \sigma(g(1))=M $. But $ y $ was arbitrary in $ G_\identity $ so we conclude that $ \sigma(yby^{-1}a)=M $ for all $ y\in G_\identity $. Finally, a simple application of Theorem \ref{Musing} now gives the result.
 	\end{proof}
 	
	\begin{remarks} 
		We conclude this section with the following observations:
		\begin{enumerate}
			\item[(i)] Notice that the condition in Theorem \ref{mainLocalization} assuming that there exists a pair $ (a^\prime,b^\prime)\in\orb(a)\times\orb(b) $ such that $ \#\sigma( b^\prime a^\prime)<\infty $ implies the following weaker statement:
			$$\#\left[\bigcap_{r\in\orb(b)}\sigma(ra)\right]<\infty.$$
			Furthermore, the above statement is sufficient to deduce the last part in the proof of Theorem \ref{mainLocalization}. 
			\item[(ii)] If we go to the proof of Lemma \ref{countable} and replace $ g(\lambda) $ by one of the following entire functions:
			$$ \lambda\mapsto f(\lambda) b \left[f(\lambda)\right]^{-1}+a, \quad  \lambda\mapsto f(\lambda) b \left[f(\lambda)\right]^{-1}-a, \quad  \lambda\mapsto f(\lambda) b \left[f(\lambda)\right]^{-1}\circ a, $$
			then we obtain Theorem \ref{mainLocalization} (in exactly the same manner) for the forms $ r+a $, $ r-a $, $ r\circ a $ respectively.
		\end{enumerate}
	\end{remarks}

	\section{Applications to Algebraic Elements}\label{sec4}
	It will be interesting to consider the results highlighted in the preceding section for elements in $ A $ with more properties—specifically—algebraic elements. For this, it will be agreed upon that, for the remainder of this section, $ p $ will be a polynomial of the form
	$$p(\lambda)\coloneqq\prod_{i=1}^{n}\left(\lambda-\lambda_i\right),$$
	where $ K=\left\{\lambda_1,\ldots,\lambda_n\right\}$ is a fixed set of distinct complex numbers. The set of elements of $ A $ which are algebraic under $ p $ will be denoted by $ E(p) $. If $ K=\{0,1\} $, then the elements of $ E(p) $ are idempotents. If $ a\in E(p) $, then one can show that $ \sigma(a)\subseteq K $. Something quite useful is the next result which decomposes algebraic elements into linear combinations of idempotents:
	\begin{theorem}[{\cite[Proposition 1]{makaiZemanek16}}]\label{decomposition}
		Let $ A $ be a Banach algebra and suppose that $ a\in A $. Then $ a\in E(p)$ if and only if there exists a set of idempotents $ \left\{e_1,\ldots,e_n\right\} $ such that $ e_ie_j=0 $ for $ i\neq j $, $ \sum_{i=1}^{n}e_i=\identity $, and 
		$$a=\sum_{i=1}^n \lambda_ie_i.$$
	\end{theorem}
	In fact, these idempotents are the Riesz idempotents obtained via the Holomorphic Functional Calculus. Hence if  $ a,b\in E(p) $ commute and their decompositions are respectively given by:
	$$a=\sum_{i=1}^n \lambda_ie_i, \quad b=\sum_{i=1}^n \lambda_i\tilde{e}_i,$$
	then $ e_i\tilde{e}_j= \tilde{e}_je_i$ for all $ i,j\in\{1,\ldots,n\} $. Of course, $ E(p) $ is not necessarily connected; in fact if $ a\in E(p) $, then the connected component of $ E(p) $ containing $ a $ is $ \orb(a) $. This was established for idempotents in \cite{z79} and for algebraic elements in \cite{makaiZemanek16}. Before we directly apply the theory obtained in Section \ref{sec1}, we establish that the product of two commuting algebraic elements is once again algebraic, via the following proposition:
	\begin{proposition}\label{myProposition}
		Let $ A $ be a semisimple Banach algebra and suppose that $ a,b\in E(p) $. If every element of $ \orb(a) $ commutes with every element of $ \orb(b) $ then $ \orb(a)\orb(b)=\orb(b)\orb(a)$ consists of algebraic elements.
	\end{proposition}
	\begin{proof}
		If we decompose $ a^\prime\in \orb(a) $ and $ b^\prime\in \orb(b) $ according to Theorem \ref{decomposition} as
		$$a^\prime=\sum_{i=1}^n \lambda_ie_i, \quad b^\prime=\sum_{i=1}^n \lambda_i\tilde{e}_i,$$
		then by direct calculation (and the aid of commutativity and orthogonality of the Riesz idempotents) we find that $$ a^\prime b^\prime  =\sum_{i=1}^n\sum_{j=1}^n\lambda_i\lambda_j e_i \tilde{e}_j,
		=\sum_{\alpha\in K^2} \alpha \hat{e}_\alpha,$$ 
		where $ \hat{e}_\alpha\hat{e}_\beta =0 $ for all $ \alpha\neq \beta $ in $ K^2 $, and $ \sum_{\alpha\in K^2}\hat{e}_\alpha=\identity $. Therefore, by Theorem \ref{decomposition}, $ a^\prime b^\prime $ is algebraic under the polynomial
		$$q(\lambda)\coloneqq \prod_{\alpha\in K^2}(\lambda-\alpha).$$
		
	\end{proof}
	It turns out that when $ a $ is algebraic, we have a few additional equivalent statements to the list in Corollary \ref{MusingPt2}. This then mostly completes generalizing {\cite[Theorem 4.3]{z79}} to algebraic elements.
	\begin{theorem}
		Let $ A $ be a semisimple Banach algebra and suppose that $ a\in E(p) $. Then the following are equivalent:
		\begin{enumerate}
			\item[\textnormal{(i)}] $ a $ is central;
			\item[\textnormal{(ii)}] there exists a polynomial $ q(\lambda) $ such that $ a E(p)\subseteq E(q) $;
			\item[\textnormal{(iii)}] there exists a polynomial $ q(\lambda) $ such that $ a \orb(a)\subseteq E(q)$;
			\item[\textnormal{(iv)}] there exists a polynomial $ q(\lambda) $ such that $ a \circ E(p)\subseteq E(q) $;
			\item[\textnormal{(v)}] there exists a polynomial $ q(\lambda) $ such that $ a \circ \orb(a)\subseteq E(q) $.
			
		\end{enumerate}
	\end{theorem}
	\begin{proof}
		If $ a $ is central, then applying an identical argument as the one given in the proof for Proposition \ref{myProposition} shows that (ii) and (iii) are true. Since $ a\circ b =\identity-(\identity-a)(\identity-b)$, it follows that (i) is also sufficient for (iv) and (v), for the subsequent reasons: Notice that if $ a $ is algebraic under the polynomial $ p $, then so too is $ \identity-a $, under 
		$$\lambda\mapsto\prod_{i=1}^n(1-\lambda_i-\lambda).$$
		Therefore, when $ b\in E(p)$, the element $ \identity-b $ is algebraic as well and hence, via commutation of $ a $ and $ b $ (as well as the proof of Proposition \ref{myProposition}), it follows that $ \identity-(\identity-a)(\identity-b) $ is algebraic under a polynomial $ q(\lambda) $. A similar case can be made showing (i)$ \Rightarrow $(v). \\
		If (ii) or (iii) is true, then obviously we must have that
		$$ \sup_{r\in \orb(a)}\rho(ar) <\infty,$$
		whence Corollary \ref{artificialList} implies that $ a $ is central. Similarly, if (iv) or (v) holds, then 
		$$ \sup_{r\in \orb(a)}\rho(a\circ r) <\infty,$$
		so that again $ a $ is central.
	\end{proof}
	For the purposes of the next results, define the following:
	$$ t_1\coloneqq \sup_{r\in \orb(b)}\rho(ar), \quad t_2\coloneqq \sup_{r\in \orb(b)}\rho(a+r),\quad \kappa\coloneqq \inf\left\{\abs{\lambda}\colon \lambda\in K^2\right\}.$$
	\begin{theorem}\label{TheBigOne2}
		Let $ A $ be a semisimple Banach algebra and suppose that $ a,b\in E(p) $. If we have that 
		$$ \#\left\{\lambda\in K^2\colon \abs{\lambda}=\kappa\right\}=1,$$
		then there exist $ \mu\geq0 $ and $ \delta>0 $ such that if $ t_1<\delta $, then
		$$\orb(a)\orb(b)=\orb(b)\orb(a)=\{\mu\identity\}.$$
		In particular, if $ a\not\in G(A) $ then 
		$$\orb(a)\orb(b)=\orb(b)\orb(a)=\{0\}.$$
	\end{theorem}
	\begin{proof}
		Choose $ \delta $ to be the real number
		$$\inf\{\abs{\lambda}\colon \lambda\in K^2, \abs{\lambda}>\kappa\}.$$
		Hence if $ t_1<\delta $, Corollary \ref{artificialList} implies that $ a^\prime b^\prime=b^\prime a^\prime $ for $ a^\prime\in \orb(a) $ and $ b^\prime\in \orb(b) $. Furthermore, $ a^\prime b^\prime $ is algebraic and $ \sigma(a^\prime b^\prime)\subseteq K^2 $. Let $ \mu $ be the complex number in $ K^2 $ satisfying $ \abs{\mu}=\kappa $. By Theorem \ref{decomposition}, if $ a^\prime b^\prime\neq \mu\identity $, it must be the case that $ \sigma(a^\prime b^\prime)\cap K^2 \neq \{\mu\} $—implying that for some $ r\in \orb(b) $ (depending on $ a^\prime $ and $ b^\prime $) we will have that $  \rho(ar)\geq \delta $, which is a contradiction. \\
		If $ a\notin G(A) $ then $ 0\in \sigma(a) $, so that $ 0\in K^2 $ and hence $ \kappa=\mu=0 $. Obviously then $ \#\left\{\lambda\in K^2\colon \abs{\lambda}=\kappa\right\}=1 $ so by the first part the orthogonality of $ \orb(a) $ and $ \orb(b) $ is apparent.
	\end{proof}
	In the special case for idempotents, we have that $ K^2=K $, $ \delta=1 $ and $ \mu=\kappa=0 $; that is, $ \orb(a) $ and $ \orb(b) $ are orthogonal when $ t_1<1 $. 
	\begin{corollary}\label{TheBigOne3}
		Let $ A $ be a semisimple Banach algebra and suppose that $ p $ and $ q $ are idempotents in $ A $. Then the following statements are equivalent:
		\begin{enumerate}
			\item[\textnormal{(i)}] $ \displaystyle \orb(p)\orb(q)=\orb(q)\orb(p)=\{0\}  $,
			\item[\textnormal{(ii)}] $\displaystyle t_1 = \sup_{r\in \orb(q)}\rho(pr)<1 $,
			\item[\textnormal{(iii)}] $ \displaystyle t_2=\sup_{r\in \orb(q)}\rho(p+r)<2 $.
		\end{enumerate}
	\end{corollary}
	\begin{proof}
		Observe that from Theorem \ref{TheBigOne2}, we immediately have that 
		$$\orb(p)\orb(q)=\orb(q)\orb(p)=\{0\} \quad \iff \quad t_1<1.$$
		Hence it suffices to show $ t_2<2 $ is equivalent to $ t_1<1 $. If $ t_2 <2 $ then Corollary \ref{artificialList} implies that every element of $\orb(p) $ commutes with every element of $ \orb(q) $. Suppose to the contrary that $ t_1\not <1 $. By commutation of $ p $ and $ r $, the element $ pr $ is an idempotent for all $ r\in \orb(q) $. But not all such elements are $ 0 $, so fix $ r\in\orb(q) $ such that $ \sigma(pr)=\{0,1\} $. Using the following result from \cite[Theorem 2.2]{BBB21},
		\begin{equation}\label{eq3}
			\lambda\not\in\{0,1\}\Rightarrow \left[\lambda\in\sigma(p+r)\iff(\lambda-1)^2\in\sigma(pr)\right], 
		\end{equation}
		
		and due to the fact that $ 2\not\in\{0,1\} $ and $ (2-1)^2\in \sigma(pr) $, it follows that $ 2\in \sigma(p+r) $—which is obviously a contradiction to the hypothesis. Thus, $ t_1<1 $. Conversely, suppose that $ t_1<1 $. Then by the first part it follows that $ \orb(p) $ and $ \orb(q) $ are orthogonal—in which case, $ p+r $ is an idempotent for all $ r\in\orb(q) $. Therefore, $ t_2=1<2 $.
	\end{proof}
	\begin{remark}
		In the proof of Theorem \ref{TheBigOne2}, the chosen $ \delta $ is a sharp upper bound for $ t_1 $. In Corollary~\ref{TheBigOne3}, $ 1 $ and $ 2 $ are sharp upper bounds for $ t_1 $ and $ t_2 $, respectively. Of course, by Corollary~\ref{artificialList}, if $t_1$ or $t_2$ is finite then the best we can do is infer that the elements of $\orb(p)$ commute with the elements of $\orb(q)$. 
	\end{remark}

	For idempotents $ p $ and $ q $ in $ A $, we have a very strong localization principle for similarity orbits. We need not even know what $ \sigma(pr) $ contains, just that it is a singleton for all $ r $ in a small open subset of $ \orb(q) $—in which case—$ p $ and $ q $ are the identity or their similarity orbits are orthogonal. 
	\begin{corollary}\label{TheSmallOne}
		Let $ A $ be a semisimple Banach algebra and suppose $ p $ and $ q $ are idempotents of $ A $ such that $ \#\sigma(pr)=1 $ for all $ r$ in an open subset $ N $ of $ \orb(q) $. Then either $ p=q=\identity $ or $ \orb(p)\orb(q)=\orb(q)\orb(p)=\{0\} $.
	\end{corollary}
	
	\begin{proof}
		We first claim that the hypothesis can be globalized in the following sense: $ \#\sigma(pr)=1 $ for all $ r$ in $ \orb(q) $. To see why this is so, fix a member $ e^{x_1}\cdots e^{x_n}q e^{-x_n}\cdots e^{-x_1} $ in $ N $ and let $$ r =e^{y_1}\cdots e^{y_m}q e^{-y_m}\cdots e^{-y_1}  $$ be arbitrary in $ \orb(q) $. Without loss of generality, we may assume that $ m=n $. Define the entire function $ f\colon \field\to A $ by:
		$$f(\lambda)\coloneqq e^{\lambda x_1+(1-\lambda)y_1}\cdots e^{\lambda x_n+(1-\lambda)y_n}.$$
		Hence by continuity of the map $ g\colon \lambda\mapsto f(\lambda)q\left[f(\lambda)\right]^{-1} $, we also have that there is some disk $ B(1,\delta) $ in $ \field $ such that $ \lambda\in B(1,\delta) $ implies $ g(\lambda)\in N $— so that $ \#\sigma(pg(\lambda)) =1$ for all $ \lambda\in  B(1,\delta)$. By the Scarcity Theorem, this must hold for all $ \lambda \in \field $, and in particular, $1= \#\sigma(pg(0)) =\#\sigma(pr)$. Hence, $ \#\sigma(pr)=1 $ for all $ r\in \orb(q) $.
		\\
		If there is some $ r\in \orb(q) $ such that $ 0\not\in\sigma(pr) $, then $ pr $ is invertible. Moreover, $ 0=(\identity-p)pr $ implies that $ p=\identity $. But then, $ r\in G(A) $ and hence $ r=q=p=\identity $. Hence if $ p\neq q $, then $$ \sigma(pwqw^{-1})=\{0\} \text{ for all } w\in G_\identity(A).$$ 
		Now observe that the value of $ t_1 $ is $ 0 $ in Corollary \ref{TheBigOne3}, so the components $ \orb(p) $ and $ \orb(q) $ are orthogonal, as advertised.
	\end{proof}
	\begin{remarks}
		A couple of remarks related to the above corollary are in order.
		\begin{enumerate}
			\item[(i)] The second part of Corollary \ref{TheSmallOne} can be directly proven without the use of Corollary \ref{TheBigOne3} with some basic representation theory. A brief outline of the proof is as follows: Begin by assuming $ pq\neq 0 $, to obtain a representation $ \pi $ on a Banach space $ X $ such that $ \pi(pq)\eta\neq 
			0 $, for some $ \eta\in X $. Then set $ \xi=\pi(q)\eta $ and argue as to why $ \{\pi(p)\xi,\pi(q)\xi\} $ is a linearly independent set. Thereafter a simple application of Sinclair's Density Theorem to this set will yield the necessary contradiction.
			\item[(ii)] If in the hypothesis of the corollary we assumed that $ \#\sigma(pr)\leq n>1$ for all $ r\in \orb(q) $, then the result would be invalid. A simple example illustrating the case for $ n=2 $ can be obtained from $ M_2(\field) $. Set
			\[p=\begin{bmatrix}
				1&0\\
				0& 0
			\end{bmatrix}, \quad q=\begin{bmatrix}
				1&0\\
				0& 1
			\end{bmatrix}.
			\]
			Then $ \#\sigma(pr)=2 $ for all $ r\in \orb(q)$ but $ p\neq q $ and $ \orb(p)\orb(q)=\orb(p)\neq\{0\} $. One may construct similar examples in $ M_n(\field) $ for larger values of $ n $.
		\end{enumerate}
	\end{remarks}
The final result gives a localization principle for algebraic elements, and, expectedly, its conclusion is stronger than that of Theorem \ref{mainLocalization}, in that we have yet another characterization of central algebraic elements. Recall that in Theorem \ref{mainLocalization}, we required in the premise that $ \sigma(b^\prime a^\prime) $ is finite for some pair $ (a^\prime,b^\prime) \in\orb(a)\times \orb(b)$. Here we are looking at just $ \orb(a) $, and obviously then $\sigma( a^2) $ is finite since $ \sigma(a) $ is finite.
	\begin{corollary}
		Let $ A $ be a semisimple Banach algebra and suppose that $ a\in A $ is algebraic. Then $ a $ is central if and only if there exist a nonempty countable set $ C \subseteq \field$ and a nonempty open subset $ N $ of $ \orb(a) $ such that $ \sigma(ra)\subseteq C $ for all $r\in N $.
	\end{corollary}
	\begin{proof}
		If $ a $ is central then $ \orb(a)=\{a\} $ and hence $ C $ may be taken as the finite set $ \sigma\left(a^2\right) $ to obtain the forward implication. \\ For the converse, we take the pair $ (a,a)\in\orb(a)\times\orb(a) $ in Theorem \ref{mainLocalization} to get that $ a $ commutes with all elements of its orbit, so that an identical argument as the one given in Corollary \ref{MusingPt2} implies the centrality of $ a $.
	\end{proof}

\end{document}